\documentclass[a4paper, reqno, 11pt]{amsart}
\usepackage{amsmath, amssymb, graphics, epsfig, enumerate, amstext, amsthm, comment}
\usepackage{xcolor}
\usepackage[noadjust]{cite}
\usepackage[plainpages=false,colorlinks,hyperindex,pdfpagemode=None,bookmarksopen,linkcolor=blue,citecolor=blue,urlcolor=blue]{hyperref}

\numberwithin{equation}{section}

\newtheorem{theorem}{Theorem}[section]
\newtheorem{proposition}[theorem]{Proposition}
\newtheorem{corollary}[theorem]{Corollary}
\newtheorem{lemma}[theorem]{Lemma}

\theoremstyle{definition}

\newtheorem{example}[theorem]{Example}
\newtheorem{remark}[theorem]{Remark}

\theoremstyle{remark}
\newtheorem{problem}[theorem]{Problem}

\def\qed{\hfill\vbox{\hrule width 6 pt
\hbox{\vrule height 6 pt width 6 pt}} \medskip}

\def\IC{{\mathbb C}}
\def\IR{{\mathbb R}}
\def\IF{{\mathbb F}}
\def\IZ{{\mathbb Z}}

\def\bM{{\bf M}}

\def\bV{{\bf V}}
\def\bW{{\bf W}}

\def\b0{{\bf 0}}

\def\bV{{\bf V}}

\def\diag{{\rm diag}\,}

\def\[{\left [}
\def\]{\right ]}
\def\({\left (}
\def\){\right )}

\def\dfrac{\displaystyle\frac}

\def\<{{\langle}}
\def\>{{\rangle}}
\def\1{{\bf 1}}

\newcommand{\T}{{\thinspace\mathrm{t}}}

\newcommand{\bprob}{\begin{problem}}
\newcommand{\eprob}{\end{problem}}

\newcounter{example}
\protected\def\verythinspace{%
  \ifmmode
    \mskip0.33\thinmuskip
  \else
    \ifhmode
      \kern0.08334em
    \fi
  \fi
}

\begin{document}
\openup .72\jot
\title[Nonsurjective zero product preservers]
{Nonsurjective zero product preservers
\\ between matrix spaces over an arbitrary field}

\author[Li, Tsai, Wang and Wong]{Chi-Kwong Li, Ming-Cheng Tsai, Ya-Shu Wang \and Ngai-Ching Wong}

\address[Li]{Department of Mathematics, The College of William
\& Mary, Williamsburg, VA 13185, USA.}
\email{ckli@math.wm.edu}

\address[Tsai]{General Education Center, National Taipei University of Technology, Taipei 10608, Taiwan.}
\email{mctsai2@mail.ntut.edu.tw}

\address[Wang]{Department of Applied Mathematics, National Chung Hsing University, Taichung 40227, Taiwan.}
\email{yashu@nchu.edu.tw}

\address[Wong]{Department of Applied Mathematics, National Sun Yat-sen
  University, Kaohsiung, 80424, Taiwan; Department of Healthcare Administration and Medical Information, and Center of
Fundamental Science, Kaohsiung  Medical University, 80708 Kaohsiung, Taiwan.}
  \email{wong@math.nsysu.edu.tw}

\date{\today}

\begin{abstract}
A map $\Phi$ between matrices
is said to be zero  product preserving if
$$
\Phi(A)\Phi(B) = 0 \quad \text{whenever}\quad AB = 0.
$$
In this paper, we give concrete descriptions of an
additive/linear zero product  preserver $\Phi: \bM_n(\mathbb{F}) \rightarrow \bM_r(\mathbb{F})$
between matrix algebras of different dimensions over an arbitrary field $\mathbb{F}$.
In particular,
we show that if $\Phi$ is linear and preserves
zero products then
$$
\Phi(A)=  S\begin{pmatrix} R_1 \otimes A & 0 \cr 0 & \Phi_0(A)\end{pmatrix} S^{-1},
$$
for some invertible matrices $R_1$ in $\bM_k(\mathbb{F})$,
$S$ in $\bM_r(\mathbb{F})$ and a zero product preserving linear
map $\Phi_0: \bM_n(\mathbb{F}) \rightarrow \bM_{r-nk}(\mathbb{F})$ into
nilpotent matrices.
If $\Phi(I_n)$ is invertible, then $\Phi_0$ is vacuous.
In general, the structure of $\Phi_0$ could be quite arbitrary, especially
when $\Phi_0(\bM_n(\IF))$ has trivial multiplication, i.e.,
$\Phi_0(X)\Phi_0(Y) = 0$ for all $X, Y$ in $\bM_n(\mathbb{F})$.
We show that if
$\Phi_0(I_n) = 0$ or  $r-nk \le n+1$, then
$\Phi_0(\bM_n(\IF))$ indeed has trivial multiplication.
More generally, we characterize  subspaces $\bV$ of square matrices satisfying
$XY = 0$ for any $X, Y \in \bV$.
Similar results for
double zero product preserving maps are obtained.
 \end{abstract}


\maketitle


\section{Introduction}

Preserver problems
of matrices and operators attract a lot of attention;  see,
for example,  \cite{Semrl06, LT92,LP,CKLW03, Wong05, W-zpp, Monlar, LCLW18, LTWW20, GLS},
and the references therein.
Some authors study those additive or linear maps  $\Phi$  preserving zero products,
that is,
$$
\Phi(A)\Phi(B) = 0\quad \text{whenever}\quad AB = 0.
$$
See, for example, \cite{ABEV09, BS-zp,Bresar12,CLT87,BFGMP08, LW13,Li02,CLP}.
The classical results of Jacobson, Rickart, Kaplansky, Herstien, etc.\ (see, e.g., \cite{JR50, H56}),
ensure  that  every \emph{surjective} zero product preserving  {linear}
 map $\Phi: \bM_n(\mathbb{F}) \to \bM_n(\mathbb{F})$ between matrices over an arbitrary
 field $\IF$ is a scalar multiple of
an inner algebra isomorphism, namely,
$A \mapsto \alpha S^{-1}AS$,
for a nonzero scalar $\alpha$ and an invertible $S$ in $\bM_n(\mathbb{F})$.
  See, e.g.,  \cite[Theorems 2.6 and 3.1]{CKLW03}.

The situation is quite different when $\Phi$ is not surjective.
For example, let $f$ be any   map sending $n\times n$ square matrices to $p\times q$ rectangular matrices over any field $\IF$.
The map $\Phi: \bM_n(\IF)\to \bM_{p+q}(\IF)$ defined by $A\mapsto \left(
                                 \begin{array}{cc}
                                   0_p & f(A) \\
                                   0 & 0_q \\
                                 \end{array}
                               \right)\in \bM_{p+q}(\IF)$
satisfies that $\Phi(A)\Phi(B) = 0_{p+q}$ for any $A, B \in \bM_n(\IF)$.
It trivially preserves zero products.
One may impose extra assumptions such as being additive or linear to $f$
so that the map $\Phi$ will be additive or linear,
respectively.
However, $\Phi$ is far away from being multiplicative in any case.

In this paper, we give  descriptions of the structures of an
additive or linear zero product preserver  $\Phi: \bM_n(\IF) \rightarrow \bM_r(\IF)$
for arbitrary dimensions $n$ and $r$.
More precisely, suppose  $S^{-1}\Phi(I_n)S = R \oplus N$
where $S \in \bM_r(\IF), R \in \bM_s(\IF)$ are invertible, and
$N \in \bM_{r-s}(\IF)$ is nilpotent with nil index $\nu$.
If $\Phi$ is additive, then $s = nk$ for some integer $k\geq 0$, and $\Phi$ assumes the form
$$
A \mapsto
      S \begin{pmatrix} R\Phi_1(A)  & 0 \cr 0 & \Phi_0(A) \cr\end{pmatrix} S^{-1},
$$
where $\Phi_1: \bM_n(\IF)\to \bM_{nk}(\IF)$ is a ring homomorphism
such that $\Phi_1(A)R = R \Phi_1(A)$ for all $A \in \bM_{n}$,
and $\Phi_0:\bM_n(\IF)\to \bM_{r-nk}(\IF)$ is an
additive zero product preserving map such that $\Phi_0(A)^{\nu+1} = 0$
for every $A$.

If $\Phi$ is linear then, by using a suitable  $S\in \bM_r(\IF)$ to write $S^{-1}\Phi(I_n)S = R\oplus N$,
we have $R=R_1 \otimes I_n$ for an $R_1 \in \bM_k(\IF)$ such that $R\Phi_1$ assumes the form
$A\mapsto (R_1 \otimes I_n)(I_k\otimes A)=R_1\otimes A$.
The map $\Phi_0$ could have wild structure.
Nevertheless, we obtain additional information on $\Phi_0$
under some mild assumptions.
In particular, if $\Phi(I_n)$ is invertible,  $\Phi$ assumes the form $A\mapsto S(R_1 \otimes A)S^{-1}$;
if $\Phi(I_n)$ is a nilpotent, we have $\Phi=\Phi_0$.
More generally, if $\Phi(I_n)$ is diagonalizable, then $N=0$ and thus $\nu=1$; in
this case, $\Phi_0(A)\Phi_0(B)=0$
for all $A,B\in \bM_n(\IF)$.

Our paper is organized as follows.
In Section \ref{s:prel}, we  collect some preliminary results.
In Section \ref{s:0p}, we study  additive and linear zero product preservers
between matrix algebras of arbitrary sizes over an arbitrary field.
We also study linear maps $\Phi$ between matrix algebras preserving  double zero
products, that is,
$$
\Phi(A)\Phi(B)=\Phi(B)\Phi(A)=0\quad\text{whenever}\quad AB=BA=0.
$$
In this case, $\Phi$ assumes the form
\begin{equation*}
A \mapsto  S
\begin{pmatrix}
R_1 \otimes A & 0 & 0 \cr
0 & R_2 \otimes A^{\verythinspace\mathrm{t}} & 0 \cr
0 & 0 & \Phi_0(A)\cr\end{pmatrix}S^{-1}
\end{equation*}
for some invertible matrices $R_1\in \bM_p(\IF), R_2\in \bM_q(\IF)$ and $S\in \bM_r(\IF)$.
Here the linear map $\Phi_0: \bM_n(\IF)\to \bM_{r-n(p+q)}(\IF)$ preserves double zero products and
sends diagonalizable matrices to nilpotent elements.
In Section \ref{s:3.3}, we put attention on  additive zero product preserving maps
$\Phi_0$ into nilpotent matrices.
In particular, we obtain sufficient and necessary conditions to ensure that the range space of such a $\Phi_0$ has
trivial multiplication.
Finally,  we study in details the structure of the  range space of $\Phi_0$ when
it has trivial multiplication.

\section{Preliminaries}\label{s:prel}

In our discussion, we always assume that $\IF$ is an arbitrary field, and
$\{E_{11}, E_{12}, \dots, E_{nn}\}$ the standard basis for the algebra $\bM_n(\IF)$ of $n\times n$ matrices over $\IF$.
In other words, $E_{ij}=e_i{e_j}^{\tiny{\textrm{t}}}$ where $\{e_1, \ldots, e_n\}$ is the standard
basis for the vector space $\IF^n$ over $\IF$, and $A^\T$ denotes the transpose of a rectangular matrix $A$.
An  {idempotent} is a matrix $E$  satisfying $E^2=E$.  Two idempotents  $E,F\in \bM_n(\IF)$ are  {disjoint} if $EF=FE=0$.
A matrix $N$ is a  {nilpotent} if $N^\nu=0$ for some positive integer $\nu$.

\begin{lemma} \label{P1P2}
Let  $n$ be a positive integer with $n \ge 2$.
\begin{enumerate}[{\rm (a)}]
\item
The linear space $\bM_n(\IF)$ has the following basis consisting
of rank one idempotents
$$
\big\{E_{jj}: 1 \le j \le n\big\} \cup \big\{E_{ii} + E_{ij}: 1 \le i \le n, i\ne j\big\}.
$$
For $i \ne j$, the matrices $E_{ii} + E_{ij}$ and $E_{jj}-E_{ij}$
are disjoint rank one idempotents.



\item  Every idempotent $A$ in $\bM_n(\IF)$ is a sum
of $k$ disjoint rank one idempotents, where $k$ is the rank of $A$.
\item Suppose the characteristic of $\IF$ is not 2.
The sum of two idempotents is an idempotent if and only if
they are disjoint.

    \item Every non-invertible matrix in $\bM_n(\mathbb{F})$ is a product of idempotents.

    \item The ring $\bM_n(\mathbb{F})$ is generated by its idempotents.

\end{enumerate}
\end{lemma}
\begin{proof}
Assertions (a)--(c) can be  verified directly.
Assertion (d)  is shown in \cite[Theorem]{Erdos67}.
Assertion (e) is a consequence of  (d) and the fact that every matrix can
be written as a sum of rank one matrices.
\end{proof}

The following result can be found in \cite[Corollary 6.7.2]{ZTC06} under the assumption that
the characteristic of $\IF$ is not   $2$ and $n\geq 3$.
Below, we give a self contained proof of the result without the restrictions.

\begin{theorem}\label{thm:ring-hom}
Let  $\Phi: \bM_n(\IF)\to \bM_r(\IF)$ be a
ring homomorphism.  Then there is an invertible $S$ in $\bM_r(\IF)$ and a unital ring
homomorphism $\tau:\mathbb{F}\to \bM_k(\IF)$ with $r-nk\geq 0$ such that
  $$
  \Phi(A) = S[(\sum_{i,j=1}^n \tau(a_{ij})\otimes E_{ij})\oplus 0_{r-nk}]S^{-1}\quad\text{for all $A=(a_{ij})\in \bM_n(\IF)$.}
  $$
\end{theorem}
\begin{proof}
 Replacing $\Phi$ with the map $X \mapsto S_1^{-1} \Phi(X)S_1$ for some
invertible $S_1$ in $\bM_r(\IF)$,
we can assume that $\Phi(I_n)=I_{r_1}\oplus 0_{r_2}$ with $r=r_1+r_2$.
Since $\Phi(A)=\Phi(I_nAI_n)=\Phi(I_n)\Phi(A)\Phi(I_n)$,
we may further assume that $\Phi(I_n)=I_r$.
Since $E_{ij}E_{kl} = \delta_{jk}E_{il}$, we have
\begin{align}\label{eq:ijkl}
  \Phi(E_{ij})\Phi(E_{kl}) = \delta_{jk}\Phi(E_{il}), \quad i,j,k,l=1,2,\ldots, n,
\end{align}
where $\delta_{ij}=1$ when $i=j$, and $0$ otherwise.
Moreover,
$$
I_r = \Phi(I_n) = \sum_{i=1}^n \Phi(E_{ii}).
$$
Replacing $\Phi$ with the map $X \mapsto S_2^{-1} \Phi(X)S_2$ for some
invertible $S_2$ in $\bM_r(\IF)$,
we can assume that the idempotents
$$
\Phi(E_{ii}) = 0_{k_1} \oplus \cdots \oplus 0_{k_{i-1}} \oplus I_{k_i} \oplus 0_{k_{i+1}}\oplus \cdots \oplus 0_{k_n}, \quad i=1,\ldots, n.
$$
Here, $k_1 + \cdots + k_n=r$.

Let $m=r-k_1 - k_2$.
It follows from \eqref{eq:ijkl} that
$$
\Phi(E_{12}) =
\begin{pmatrix} B_{11} & B_{12} \cr
B_{21} & B_{22} \cr \end{pmatrix}\oplus 0_m
\quad\hbox{and}\quad
\Phi(E_{21}) = \begin{pmatrix} C_{11} & C_{12}  \cr
C_{21} & C_{22} \cr\end{pmatrix}\oplus 0_m,
$$
where $B_{ij}, C_{ij}$  are $k_i\times k_j$ matrices for $i,j=1,2$.
Since $E_{11}E_{12} = E_{12}$ and $E_{12}E_{11} =0$, we have $B_{11}$, $B_{22}$ and $B_{21}$ are all zero matrices.
Similarly, $C_{11}$, $C_{22}$ and $C_{12}$ are also zero matrices.
Hence,
\begin{align}\label{eq:1221}
\Phi(E_{12}) =
\begin{pmatrix} 0 & B_{12} \cr
0 & 0 \cr \end{pmatrix}\oplus 0_m
\quad\hbox{and}\quad
\Phi(E_{21}) = \begin{pmatrix} 0 & 0  \cr
C_{21} & 0 \cr\end{pmatrix}\oplus 0_m,
\end{align}
On the other hand,
$(E_{12}+E_{21})^2 = E_{11} + E_{22}$
implies
$$
\begin{pmatrix} 0 & B_{12} \cr
C_{21} & 0 \cr\end{pmatrix}^2
=
\begin{pmatrix} B_{12}C_{21} & 0  \cr
0 & C_{21}B_{12} \cr\end{pmatrix}
=
\begin{pmatrix} I_{k_1} & 0  \cr
0 & I_{k_2}  \cr\end{pmatrix}.
$$
This ensures $k_1=k_2$ and $B_{12} = C_{21}^{-1}$.
Let $k=k_1$.

Applying  a similar argument to other $(i,j)$ pairs, we see that
$$
\Phi(E_{jj}) = E_{jj} \otimes I_k,\quad
\Phi(E_{ij}) = E_{ij} \otimes B_{ij}\  \text{for}\ i < j,\quad
\Phi(E_{ij}) = E_{ij} \otimes B_{ji}^{-1}\ \text{for}\ j < i.$$
In particular, $r/n =k$.

Replacing $\Phi$ by the map $X \mapsto S_3\Phi(X)S_3^{-1}$
with $S_3 = I_k \oplus B_{12} \oplus B_{13} \oplus  \cdots \oplus B_{1n}$, we can further assume that
$$
B_{12} = \cdots = B_{1n} = I_k.
$$
Actually, if $n\geq 3$ we also have
$$
\Phi(E_{ij}) = E_{ij}\otimes I_k \quad
\hbox{ for all } i,j = 1,\ldots, n.
$$
To see this, observe $E_{ij}=(E_{i1} + E_{1j} + E_{ij})^2$ for $1<i<j$.  We thus have
$$
\Phi(E_{ij}) = (\Phi(E_{i1}) + \Phi(E_{1j}) + \Phi(E_{ij}))^2.
$$
This gives
$$
E_{ij} \otimes B_{ij} = (E_{i1} \otimes I_k + E_{1j} \otimes I_k + E_{ij} \otimes B_{ij})^2
= E_{ij} \otimes I_k.
$$

Reordering the basic vectors, i.e., applying a permutation similarity, we can assume
instead
\begin{align}\label{eq:stand-Eij}
\Phi(E_{ij}) = I_k\otimes E_{ij} \quad
\hbox{ for all } i,j = 1,\ldots, n.
\end{align}
  For any $a$ in $\mathbb{F}$, the matrix $\Phi(aI_n)$ commutes with all $\Phi(E_{ij})= I_k\otimes E_{ij}$.
  Thus,  $\Phi(aI_n)=\tau(a)\otimes I_n$ for some $\tau(a)\in \bM_k(\IF)$.
   It is easy to see that $a\mapsto \tau(a)$ is a unital ring homomorphism from $\mathbb{F}$ into $\bM_k(\IF)$.
  Consequently,
  $$
  \Phi(A)= \sum_{i,j=1}^{n} \Phi(a_{ij}E_{ij})=\sum_{i,j=1}^{n} \Phi(a_{ij}I_n)\Phi(E_{ij})
  = \sum_{i,j=1}^{n} \tau(a_{ij})\otimes E_{ij}.
  $$
 \vskip -0.45in
\end{proof}

When the ring homomorphism  $\Phi$ in Theorem \ref{thm:ring-hom} is also linear, we see that
\begin{align*}
  \Phi(A)&= \sum_{i,j=1}^{n} \Phi(a_{ij}E_{ij})=\sum_{i,j=1}^{n} a_{ij}\Phi(E_{ij})\\
  &= \sum_{i,j=1}^{n} a_{ij}I_k\otimes E_{ij} = I_k\otimes A, \quad\forall A=(a_{ij}) \in \bM_n(\IF).
\end{align*}
Consequently, we have

\begin{theorem}
\label{thm:alg-homo}\label{main2.5}
Suppose $\Phi: \bM_n(\IF) \rightarrow \bM_r(\IF)$
is an algebra homomorphism.
Then there exist a nonnegative integer $k$ with $t = r-nk \geq 0$,
and  an invertible matrix $S$ in $\bM_r(\IF)$ such that
 $\Phi$  assumes the form
\begin{align}\label{unital-form}
A \mapsto S
\left(
  \begin{array}{cc}
    I_k \otimes A& 0\\
    0 & 0_t \\
  \end{array}
\right)S^{-1}.
\end{align}
\end{theorem}

One might expect
that a ring homomorphism $\Phi$ between matrices assumes a form
similar to \eqref{unital-form}; namely,
 \begin{align}\label{eq:goodringform}
  A\mapsto \alpha S(A_{\tau_1} \oplus \cdots\oplus A_{\tau_k} \oplus 0) S^{-1}
  \end{align}
for some unital ring endomorphisms $\tau_1, \ldots, \tau_k$ of $\IF$.
Here, $A_\tau$ denotes the matrix $(\tau(a_{ij}))$ when $A=(a_{ij})$.
If \eqref{eq:goodringform} holds and  $\Phi$ is also linear, then all $\tau_k$ are the identity map and \eqref{eq:goodringform}
reduces to \eqref{unital-form}.
However, the following  example tells that it is not always the case.

\begin{example}\label{eg:ringhom}
 Let $\mathbb{F}$ be a purely transcendental extension over another field $\mathbb{K}$, for example $\mathbb{R}/\mathbb{Q}$.
  According to \cite[Corollary 1' in p.~124]{SZbook75},
there is  a nonzero additive derivation $x\mapsto x'$  of  $\mathbb{F}$.
Consider the unital ring homomorphism $\tau: \mathbb{F} \to \bM_2(\mathbb{F})$ defined by
\begin{align}\label{eq:derivation}
                        \tau(a)= \left(\begin{array}{cc} a & a' \\ 0 & a \\ \end{array}\right).
\end{align}
  Note that $\tau(a)$ is not diagonalizable whenever $a'\neq 0$.
  Consequently, any ring homomorphism $\Phi:\bM_n(\IF)\to \bM_r(\IF)$ with $r\geq 2n$ defined by $\tau$ as in Theorem \ref{thm:ring-hom}
  does not
  assume the form \eqref{eq:goodringform}.
\end{example}

\section{Additive and Linear Maps Preserving Zero Products}\label{s:0p}


We study those additive/linear maps $\Phi: \bM_n(\IF)\rightarrow \bM_r(\IF)$
preserving zero products, that is,
$$\Phi(A)\Phi(B) = 0_r \qquad \hbox{ whenever } A, B \in \bM_n(\IF) \hbox{ satisfy }
AB = 0_n.
$$
We need the following result, which is known as the \emph{fitting decomposition}.

\begin{lemma}[{See, e.g., \cite[Theorem A.0.4] {ZTC06}}]\label{nil}
Every $A \in \bM_n(\mathbb{F})$ is similar to a direct sum
 $R\oplus N$  of an invertible matrix $R\in \bM_s(\IF)$  and a nilpotent matrix  $N\in \bM_{n-s}(\IF)$ such that
$N$ is a direct sum of upper triangular Jordan blocks
for the eigenvalue zero of $A$. Here, $R$ or $N$ can be vacuous.
In particular, if $A$ is diagonalizable (resp.\ an idempotent), then $N = 0_{n-s}$ (resp.\ and $R = I_s$).
\end{lemma}

By Lemma \ref{nil}, there is an invertible matrix $S$ in $\bM_r(\IF)$ such that
$$
S^{-1}\Phi(I_n)S = R \oplus N,
$$
where $R$ in $\bM_s(\IF)$ is invertible, and $N$ in $\bM_{r-s}(\IF)$ is nilpotent such that
$N$ is a direct sum of upper triangular  Jordan blocks
 for the eigenvalue zero of $\Phi(I_n)$.
Furthermore, the size $\nu$ of the largest  Jordan block of $N$ is
the \emph{nil index} of the nilpotent matrix $N$, which
is the smallest positive integer $\nu$ such that $N^\nu=0$.
If $S_1^{-1}\Phi(I_n)S_1 = R_1\oplus N_1$ is another direct sum of an invertible matrix $R_1$ and
a nilpotent matrix $N_1$ for an invertible matrix $S_1$,
then for any $k\geq1$,
$$
S(R^k\oplus N^k)S^{-1}=(S(R\oplus N)S^{-1})^k = (S_1(R_1\oplus N_1)S_1^{-1})^k=S_1(R_1^k\oplus N_1^k)S_1^{-1}.
$$
Since $R, R_1, S, S_1$ are all invertible, by counting  ranks we conclude that
the nilpotent matrices $N$ and $N_1$ have the same Jordan form.
Therefore, we see that   $N_1, N$ have the same nil index $\nu$, and $R_1, R$ have the same rank $s$.
It is clear that $s$ is the rank of the $r\times r$ matrix $\Phi(I_n)^r$.
When the nil index $\nu=1$, that is $N=0$, we say that $\Phi(I_n)$ \emph{has a zero nilpotent part};
in this case, all the zero Jordan blocks of
$\Phi(I_n)$ are $1\times 1$.

\begin{theorem}\label{main}
Suppose  $\Phi: \bM_n(\IF) \rightarrow \bM_r(\IF)$ is
an additive map preserving zero products. Then
there are invertible matrices  $S$ in $\bM_r(\IF)$ and $R_1$ in $\bM_k(\IF)$
such that $\Phi$ has the form
\begin{equation} \label{main-form}
A \mapsto
      S \begin{pmatrix} (R_1\otimes I_n)\Phi_1(A)  & 0 \cr 0 & \Phi_0(A) \cr\end{pmatrix} S^{-1}.
\end{equation}
Here,
$\Phi_1: \bM_n(\IF) \rightarrow \bM_{nk}(\IF)$ has the form
$(a_{ij}) \mapsto \sum_{ij} \tau(a_{ij}) \otimes E_{ij} \in \bM_k \otimes \bM_n$
for a unital
ring homomorphism $\tau: \IF \rightarrow \bM_k(\IF)$, such that
$R_1\tau(a) = \tau(a)R_1$ for all $a \in \IF$,
and  $\Phi_0: \bM_n(\IF) \rightarrow \bM_{r-kn}(\IF)$ is a zero product preserving
additive map  into nilpotent matrices such that
$\Phi_0(I_n)$ commutes with $\Phi_0(A)$ for all $A\in\bM_n(\IF)$,
and the product of any $\nu+1$  elements in
$\Phi_0(\bM_n(\IF))$ is zero if $\Phi_0(I_n)$
has nil index $\nu$.

In particular, $\Phi(I_n) = S[(R_1 \otimes I_n) \oplus \Phi_0(I_n)]S^{-1}$,
and the following hold.
\begin{itemize}
\item[{\rm (1)}]  $\Phi(I_n)$ is invertible if and only if
$\Phi_0$ is vacuous.
\item[{\rm (2)}] $\Phi(I_n) = I_r$ if and only if $\Phi_0$ is vacuous
and $R_1 = I_k$.
\item[{\rm (3)}] $\Phi(I_n)$ is nilpotent if and only if $\Phi_1$ is vacuous.
\item[{\rm (4)}] If $\Phi(I_n)$ has a zero nilpotent part, i.e., $\Phi_0(I_n) = 0$, then
$\Phi_0(X)\Phi_0(Y) = 0$ for all $X, Y \in \bM_n(\IF)$.
\end{itemize}
\end{theorem}

We need the following elementary lemma to prove Theorem \ref{main}.
We sketch a proof for easy reference.

\begin{lemma}\label{lem:zp}
Suppose  $\Phi: \bM_n(\IF) \rightarrow \bM_r(\IF)$ is an additive map preserving zero products.
Then
$$
\Phi(C)\Phi(AB)=\Phi(CA)\Phi(B) \quad \hbox{ for all } \
 A, B, C \in \bM_n(\IF).
$$
Consequently,
\begin{equation} \label{iab=ab}
\Phi(I_n)\Phi(AB)=\Phi(A)\Phi(B) \quad \hbox{ for all } \
A, B \in \bM_n(\IF),
\end{equation}
and
\begin{equation} \label{ia=ai}
\Phi(I_n) \Phi(A) = \Phi(A)\Phi(I_n) \quad \hbox{ for all } \ A \in \bM_n(\IF).
\end{equation}
\begin{enumerate}[\rm (a)]
    \item If $\Phi(I_n)$ is invertible then
$A\mapsto \Phi(I_n)^{-1}\Phi(A)$ is a ring homomorphism from $\bM_n(\IF)$ into $\bM_r(\IF)$.
    \item If $\Phi(I_n)^\nu=0$
    then the product of any
    $\nu+1$ elements from the range of $\Phi$ is zero, i.e.,
    $$
    \Phi(A_1)\Phi(A_2)\cdots\Phi(A_{\nu+1})=0
    \quad\hbox{ for all } \ A_1,A_2,\ldots, A_{\nu+1}\in \bM_n(\IF).
    $$
    In particular, if $\Phi(I_n)=0$ then
    the range of $\Phi$ has trivial multiplication, i.e.,
    $$
    \Phi(A)\Phi(B)=0 \quad\hbox{ for all } \ A,B\in \bM_n(\IF).
    $$
\end{enumerate}
\end{lemma}
\begin{proof}
We  follow  the proof of \cite[Lemma 2.1]{CKLW03}.
The case $n=1$ is obvious.  Assume below that $n\geq 2$.
Let $E=E^2$ in $\bM_n(\IF)$.  For any $B, C$ in $\bM_n(\IF)$, consider
$$
(C - CE)EB = CE(B-EB)=0.
$$
By the zero product preserving property, we have
$$
(\Phi(C) - \Phi(CE))\Phi(EB)=\Phi(CE)(\Phi(B)-\Phi(EB))=0.
$$
It follows
$$
\Phi(C)\Phi(EB)=\Phi(CE)\Phi(EB)= \Phi(CE)\Phi(B).
$$
Since  $\bM_n(\IF)$ is generated by its idempotents as a ring by
Lemma \ref{P1P2},
$$
\Phi(C)\Phi(AB)=\Phi(CA)\Phi(B), \quad A,B,C\in \bM_n(\IF).
$$
Putting $C=I$, and putting $B = C = I$, respectively, we establish \eqref{iab=ab} and \eqref{ia=ai}.
It thus follows (a).

For (b), in view of \eqref{iab=ab} and the assumption $\Phi(I_n)^\nu=0$, we have
\begin{align*}
&\ \Phi(A_1)\Phi(A_2)\Phi(A_3)\cdots
\Phi(A_{\nu+1})
= \Phi(I_n)\Phi(A_1A_2)\Phi(A_3)\cdots\Phi(A_{\nu+1})
= \cdots \\
= &\ \Phi(I_n)^\nu\Phi(A_1A_2A_3\cdots A_{\nu+1})=0
\quad\hbox{ for all } \ A_1,A_2,\ldots, A_{\nu+1}\in \bM_n(\IF).
\end{align*}
\vskip -.2in\end{proof}

\begin{proof}[Proof of Theorem \ref{main}]
By Lemma \ref{nil},
we may assume that $\Phi(I_n)
 =  S(R\oplus N)S^{-1}$ for invertible matrices $R\in \bM_s(\IF)$, $S\in \bM_r(\IF)$, and
a nilpotent matrix $N\in \bM_{r-s}(\IF)$.
We may replace $\Phi$ by $S^{-1}\Phi(\cdot)S$ and assume that $\Phi(I_n)=R\oplus N$.
Let
$$
\Phi(X) =
\begin{pmatrix} Y_{11} & Y_{12} \cr Y_{21} & Y_{22} \cr \end{pmatrix},
$$
where $Y_{11}\in \bM_s(\IF)$. By \eqref{ia=ai} in Lemma \ref{lem:zp}, $\Phi(I_n)\Phi(X) = \Phi(X)\Phi(I_n)$.
Therefore,
$$
RY_{11}=Y_{11}R,\quad RY_{12} = Y_{12}N,\quad   NY_{21} = Y_{21} R \quad\text{and}\quad NY_{22}=Y_{22}N.
$$
Without loss of generality, we can assume that
$N  = \sum_{j} d_j E_{j,j+1}$ with $d_j \in \{0,1\}$
is a direct sum of
upper triangular Jordan blocks of $N$ with  zero diagonals.
Write $Y_{12}  = \big[v_1\,| \cdots |\, v_{r-s}\big]$, where
$v_1, \dots, v_{r-s}$ are column vectors. Then
$$
\bigg[Rv_1\,|\, Rv_2\,|\, \cdots |\, Rv_{r-s}\bigg] = \bigg[0\,|\, d_1v_1\,| \cdots |\, d_{r-s-1}v_{r-s-1}\bigg].
$$
Thus, $v_1 = R^{-1} 0 = 0$ and
$v_j = d_{j-1} R^{-1}v_{j-1} = 0$ for $j = 2, \dots, r-s$.
Hence, $Y_{12}=0$.
Similarly, we can show that $Y_{21} = 0$.
Therefore,  $\Phi(X)$ assumes the form $Y_{11} \oplus Y_{22}$.
Bringing back the similarity transformation, we can set up  the additive  maps
$\Phi_1: \bM_n(\IF) \rightarrow \bM_s(\IF)$ and $\Phi_0: \bM_n(\IF) \rightarrow \bM_{r-s}(\IF)$ such that
\begin{equation*} 
 S^{-1}\Phi(X)S = R\Phi_1(X) \oplus \Phi_0(X).
 \end{equation*}
Clearly, $\Phi_1(I_n) = R^{-1}R = I_s$.
Moreover, $R\Phi_1(A)=\Phi_1(A)R$ for all $A$ in $\bM_n(\IF)$.
Suppose $A, B \in \bM_n(\IF)$ such that $AB=0_n$.
Let  $S^{-1}\Phi(A)S =  A_1 \oplus A_2$
and $S^{-1}\Phi(B)S = B_1 \oplus B_2$.
Since $\Phi(A)\Phi(B)=0_r$, we have $A_1 B_1 = 0_s$ and $A_2 B_2=0_{r-s}$.
It also follows
$$ R^{2}\Phi_1(A)\Phi_1(B) = R\Phi_1(A)R\Phi_1(B)=A_1B_1 = 0_s.$$
 Consequently,
both $\Phi_1, \Phi_0$ preserve zero products as well.
By Lemma \ref{lem:zp},
$\Phi_1$ is a unital  ring homomorphism,
and  $\Phi_0$ satisfies the said conclusion.

Next, we show that $R$   assumes the form $R_1 \otimes I_n$.
Assume that $\tau$ is a unital ring
homomorphism such that $\Phi_1(A) = \sum_{i,j=1}^n \tau(a_{ij}) \otimes E_{ij}$ for
$A = (a_{ij})\in \bM_n(\IF)$.
Let $P\in \bM_{nk}(\IF)$ be the permutation matrix
such that $P(X\otimes Y)P^t = Y\otimes X$ for  $(X,Y)
\in  \bM_k(\IF) \times \bM_n(\IF)$.
Then  $P\Phi_1(E_{ij})P^t = E_{ij} \otimes  \tau(1) = E_{ij} \otimes I_k$, and thus
  $R \Phi(E_{jj}) = \Phi(E_{jj}) R$, for all $i,j=1,\ldots, n$,
implies that $PR P^t = R_{11} \oplus \cdots \oplus R_{nn}$
with $R_{jj} \in \bM_k(\IF)$ for $j = 1, \dots, n$.
Furthermore, the fact  $R\Phi(E_{1j}) = \Phi(E_{1j}) R$
for $j \ne 1$  implies that  $R_{ii} = R_{jj}:= R_1$.

One can readily verify assertions  (1) -- (4).
\end{proof}

\begin{theorem} \label{main2}\label{unital-homo}
Let  $\Phi: \bM_n(\IF) \rightarrow \bM_r(\IF)$ be a linear map preserving zero products.
Then $\Phi$ assumes the form
\begin{equation*}
A \mapsto S
\left(
  \begin{array}{cc}
    R_1 \otimes A & 0\\
    0 & \Phi_0(A) \\
  \end{array}
\right)S^{-1}
\end{equation*}
for some invertible matrices $S\in \bM_r(\IF)$ and $R_1\in\bM_k(\IF)$ with $r-nk\geq 0$,
and a zero product preserving linear map $\Phi_0: \bM_n(\IF)\to \bM_{r-nk}(\IF)$ into nilpotent matrices,
as stated in Theorem {\rm \ref{main}}.
Moreover, if $\Phi$ sends rank one  idempotents to idempotents
then $\Phi_0$  is the zero map.
\end{theorem}


\begin{proof}
The first assertion follows directly from Theorem  \ref{main}.
Assume $\Phi$ sends rank one idempotents to idempotents.
Then $\Phi_0(E)=\Phi_0(E)^\nu =0$ for every rank one idempotent $E$ in $\bM_n(\IF)$,
where $\nu\geq1$ is the nil index of $\Phi_0(I_n)$.
Since every matrix is a linear combination of rank one idempotents by
Lemma \ref{P1P2}, we see that $\Phi_0$ is the zero map.
\end{proof}

\begin{example}\label{eg:Idemp}\label{eg:counter-symmetric}
The condition that $\Phi$ sending rank one idempotents to idempotents in Theorem \ref{main2}  cannot be replaced by the
weaker one that $\Phi(I_n)$ being an idempotent.
  Consider the map $\Phi: \bM_n(\mathbb{F}) \to \bM_{r}(\mathbb{F})$ defined by
$$
\Phi(A)=A\oplus \left[(a_{11}-a_{22})N
\right],
                                \quad\text{for all $A=(a_{ij})\in \bM_n(\mathbb{F})$,}
$$
with any field $\IF$,  any $r-1>n> 1$, and any
$N\in \bM_{r-n}(\IF)$ with $N^2=0$.
The injective linear map $\Phi$ preserves zero products,
and $\Phi(I_n)=I_n\oplus 0_{r-n}$ is an idempotent.
Although the range of $\Phi_0(A)= (a_{11}-a_{22})N
$
has trivial multiplication, $\Phi_0\neq 0$.

%
\end{example}

We can use the above results and techniques
to study linear    double zero product preservers, i.e., those maps $\Phi$ between matrices such that
$$
\Phi(A)\Phi(B) = \Phi(B)\Phi(A) = 0  \quad\text{whenever}  AB = BA = 0.
$$
To this end, we also need the following result in {\cite{LWWT-Idem}}.

\begin{theorem}\label{thm:J-homo}
Let $\Phi: \bM_n(\IF) \rightarrow \bM_r(\IF)$ be a linear map.
Assume that $\Phi$ sends disjoint rank one idempotents to
disjoint idempotents.
\begin{enumerate}[\rm (a)]
  \item Suppose  $\bM_n(\IF)\neq \bM_2(\mathbb{Z}_2)$. Then there are nonnegative integers $p, q$ with $s = n(p+q) \le r$,
and an invertible matrix $S$ in $\bM_r(\IF)$ such that
$\Phi$ assumes the form
\begin{align}\label{oldform-2}
A \mapsto
S\left(
  \begin{array}{ccc}
    I_{p}\otimes A&   &     \\
      &   I_{q}\otimes A^\T &     \\
      &   & 0_{r-s} \\
  \end{array}
\right)S^{-1} \quad\text{for all $A\in \bM_n(\IF)$.}
\end{align}
  \item Suppose  $\bM_n(\IF)= \bM_2(\mathbb{Z}_2)$. Then
  there are  nonnegative integers $k_1, k_2$ with $k_1 +k_2 \leq r$ and an invertible matrix $S$ in $\bM_r(\IZ_2)$ such
    that $\Phi$ assumes the form
\begin{align*}
\begin{pmatrix}
  a & b \\
  c & d \\
\end{pmatrix}
  \mapsto
S\left[\begin{pmatrix}
  aI_{k_1}+bB_{11}+cC_{11} & bB_{12} + cC_{12} \\
  bB_{21} + cC_{21} & dI_{k_2}+bB_{22}+cC_{22}  \\
\end{pmatrix}\oplus 0_{r-k_1-k_2}\right]S^{-1},
\end{align*}
where $B_{ij}, C_{ij}$ are rectangular $k_i\times k_j$ matrices for $i,j=1,2$  satisfying that
\begin{align*}
 \begin{pmatrix}
   B_{11} & B_{12} \\
   B_{21} & B_{22}  \\
 \end{pmatrix}^2   =
 \begin{pmatrix}
   B_{11} & 0  \\
   0 & B_{22}  \\
 \end{pmatrix}
 \quad\text{and}\quad
 \begin{pmatrix}
   C_{11} & C_{12}   \\
   C_{21} & C_{22}   \\
 \end{pmatrix}^2  =
 \begin{pmatrix}
   C_{11} & 0  \\
   0 & C_{22}   \\
 \end{pmatrix}.
\end{align*}
\end{enumerate}
Conversely, if   $\Phi$ assumes the stated form in either case, then $\Phi$ sends disjoint rank one idempotents to disjoint idempotents.
\end{theorem}

\begin{theorem}\label{DZP-main}
Let  $\Phi: \bM_n(\IF) \rightarrow \bM_r(\IF)$ be a  linear map
preserving double zero products, i.e.,
$$
\Phi(A)\Phi(B) = \Phi(B)\Phi(A) = 0  \quad\text{whenever}\quad A, B \in \bM_n(\IF) \ \text{satisfies}\ AB = BA = 0.
$$
Then
there exist nonnegative integers $p, q$
such that $\Phi(I_n)^r$ has rank $s = n(p +q)$, and invertible matrices
$R_1$ in $\bM_{p}(\IF)$ and $R_2$ in $\bM_{q}(\IF)$ such that $\Phi$ assumes the form
\begin{equation*}
A \mapsto  S
\begin{pmatrix}
R_1 \otimes A & 0 & 0 \cr
0 & R_2 \otimes A^{\verythinspace\mathrm{t}} & 0 \cr
0 & 0 & \Phi_0(A)\cr\end{pmatrix}S^{-1}.
\end{equation*}
Here, $\Phi_0: \bM_n(\IF)\to \bM_{r-s}(\IF)$ is a double zero product preserving linear map satisfying that
$\Phi_0(I_n)$ commutes with $\Phi_0(A)$ for all $A\in\bM_n(\IF)$ and
$\Phi_0(A)$ is a nilpotent matrix
for every diagonalizable  $A$ in $\bM_n(\IF)$.

When $\Phi(I_n)$ is invertible,  $\Phi$ assumes the form
$$
A\mapsto S[(R_1\otimes A)\oplus (R_2\otimes A^\T)]S^{-1};
$$
when $\Phi$ sends rank one idempotents to idempotents,
 $\Phi$ assumes the form
$$
A\mapsto S[(R_1\otimes A)\oplus (R_2\otimes A^\T)\oplus 0]S^{-1}.
$$
\end{theorem}
\begin{proof}
Observe that for any  idempotent $E$ in $\bM_n(\IF)$, we have
$$
E(I_n-E) =(I_n-E)E =0.
$$
Thus
$$
\Phi(E)(\Phi(I_n)-\Phi(E)) = (\Phi(I_n)-\Phi(E))\Phi(E)=0.
$$
This gives
\begin{align}\label{eq:PI_n}
\Phi(E)\Phi(I_n) = \Phi(E)^2 = \Phi(I_n)\Phi(E).
\end{align}
Since every $A$ in $\bM_n(\IF)$ is a linear combination of  idempotents by Lemma \ref{P1P2},
\begin{align}\label{eq:PI_n-central}
\Phi(I_n)\Phi(A)  = \Phi(A)\Phi(I_n) \quad
\hbox{ for all } A \in \bM_n(\IF).
\end{align}

As argued in the proof of Theorem \ref{main}, we write
$$
S^{-1}\Phi(\cdot)S=R\Phi_1\oplus\Phi_0
$$
in which $R\in \bM_s(\IF)$ is invertible and
 $\Phi_1$ is a unital linear map from $\bM_n(\IF)$ into $\bM_s(\IF)$ such that $R\Phi_1(A)=\Phi_1(A)R$ for all $A\in \bM_n(\IF)$,
 while $\Phi_0: \bM_n(\IF)\to \bM_{r-s}(\IF)$ is linear.  Both $\Phi_1, \Phi_0$ preserve double zero products.
 In particular, it follows
 from \eqref{eq:PI_n-central} that  $\Phi_1(I_n)$ and $\Phi_0(I_n)$ commute with $\Phi_1(A)$ and $\Phi_0(A)$ for all $A\in\bM(\IF)$, respectively.

Since $\Phi_0(I_n)^\nu=0$ for some positive integer $\nu$,
it follows from \eqref{eq:PI_n} that
$\Phi_0(E)^{\nu +1} = \Phi_0(I_n)^{\nu}\Phi_0(E)=0$
for all idempotents $E$ in $\bM_n(\IF)$.
Let $A\in \bM_n(\IF)$ be diagonalizable, namely, $A=\sum_j a_j E_j$ is a linear sum of disjoint idempotents.
It follows from the double zero product preserving property of the linear map $\Phi_0$ that
$$
\Phi_0(A)^{\nu+1}= (\sum_j a_j \Phi_0(E_j))^{\nu+1} = \sum_j a_j^{\nu+1}\Phi_0(E_j)^{\nu+1} = 0.
$$

On the other hand, we see from \eqref{eq:PI_n} that  $\Phi_1$ preserves idempotents, and indeed,
$\Phi_1$  sends disjoint idempotents to disjoint idempotents.
Suppose that $\bM_n(\IF)= \bM_2(\IZ_2)$.
By Theorem \ref{thm:J-homo}(b),
$\Phi_1$ assumes the form
\begin{align*}
\begin{pmatrix}
  a & b \\
  c & d \\
\end{pmatrix}
  \mapsto
\begin{pmatrix}
  aI_{k_1}+bB_{11}+cC_{11} & bB_{12} + cC_{12} \\
  bB_{21} + cC_{21} & dI_{k_2}+bB_{22}+cC_{22}  \\
\end{pmatrix}.
\end{align*}
Since $E_{12}^2=E_{21}^2=0$ and $\Phi_1$ preserves square zero matrices,
\begin{gather*}
\Phi_1(E_{12})=
\begin{pmatrix}
  0_{k_1} & B_{12} \\
  B_{21} & 0_{k_2}  \\
\end{pmatrix}, \quad
\Phi_1(E_{21})=
\begin{pmatrix}
  0_{k_1} & C_{12} \\
  C_{21} & 0_{k_2}  \\
\end{pmatrix}, \\
B_{12}B_{21} = C_{12}C_{21}= 0_{k_1} \quad\text{and}\quad
B_{21}B_{12} = C_{21}C_{12}= 0_{k_2}.
\end{gather*}
Since
$(E_{11}+E_{12}+E_{21}+E_{22})^2=0$, we have
$$
\begin{pmatrix}
  I_{k_1} & B_{12} + C_{12} \\
  B_{21} + C_{21} & I_{k_2}  \\
\end{pmatrix}^2=0,
$$
and thus
\begin{align*}
I_{k_1} &= (B_{12} + C_{12})(B_{21} + C_{21}), \\
I_{k_2} &= (B_{21} + C_{21})(B_{12} + C_{12}).
\end{align*}
It follows $k_1=k_2:=k$ and $(B_{12} + C_{12})= (B_{21} + C_{21})^{-1}$.

Define $\Psi: \bM_2(\IZ_2)\to \bM_s(\IZ_2)$ by
$$
\Psi(\cdot) = \big(I_k\oplus (B_{12} + C_{12})\big)\Phi_1(\cdot)\big(I_k\oplus (B_{12} + C_{12})^{-1}\big).
$$
Then
\begin{alignat*}{2}
  \Psi(E_{11})  &=
\begin{pmatrix}
  I_{k} & 0\\
  0 & 0  \\
\end{pmatrix}, && \quad
  \Psi(E_{22}) =
\begin{pmatrix}
  0 & 0\\
  0 & I_{k}  \\
\end{pmatrix}, \\
  \Psi(E_{12}) &=
\begin{pmatrix}
  0 & B'_{12}\\
  B'_{21} & 0  \\
\end{pmatrix}, && \quad
  \Psi(E_{21}) =
\begin{pmatrix}
  0 & C'_{12}\\
  C'_{21} & 0  \\
\end{pmatrix}
\end{alignat*}
for some $B'_{12}, B'_{21}, C'_{12}, C'_{21}\in \bM_k(\IZ_2)$ such that
\begin{gather*}
  \Psi(E_{12} + E_{21}) =
\begin{pmatrix}
 0 & B'_{12} + C'_{12} \\
  B'_{21} + C'_{21} & 0  \\
\end{pmatrix}
=
\begin{pmatrix}
 0 & I_{k} \\
  I_k & 0  \\
\end{pmatrix}, \quad\text{and} \\
B'_{12}B'_{21} = C'_{12}C'_{21}=
B'_{21}B'_{12} = C'_{21}C'_{12}= 0_{k}.
\end{gather*}
By Lemma \ref{nil}, there is an invertible $U\in \bM_k(\IZ_2)$ such that
$$
 UB'_{12}U^{-1}=\begin{pmatrix} R_{12} & 0 \cr
0 & N_{12} \cr\end{pmatrix}
$$
for an invertible $R_{12}\in \bM_p(\IZ_2)$ and a nilpotent $N_{12}\in \bM_q(\IZ_2)$ with $p+q=k$.
Since $B'_{12}B'_{21}=B'_{21}B'_{12}=0_k$, we see that
$UB'_{21}U^{-1} = 0_{p} \oplus T$ with $T \in \bM_q(\IZ_2)$
satisfying $TN_{12} = N_{12}T = 0_q$.
Since $B'_{12} + C'_{12} = B'_{21} + C'_{21} = I_k$,
we have
$$
UC'_{12}U^{-1} = (I_p - R_{12}) \oplus (I_q - N_{12}) \quad \hbox{
and } \quad  UC'_{21}U^{-1} = I_p \oplus (I_q - T).
$$
Since
$(I_q-N_{12}) \in \bM_q(\IZ_2)$ is invertible, and
$$
0_k = C'_{12}C'_{21} = (UC'_{12}U^{-1})(U C'_{21} U^{-1})=
(I_p-R_{12})I_p \oplus (I_q-N_{12})(I_q-T),$$
we see that $R_{12} = I_p$ and $T=I_q$, and thus $N_{12} = 0_q$.

Let $\Psi_1(\cdot) = (U\oplus U)\Psi(\cdot)(U\oplus U)^{-1}$.  Then,
\begin{align*}
\Psi_1(E_{11})
&=
\begin{pmatrix}
I_{p} & 0 & 0_p  & 0\cr
0  & I_{q} & 0 & 0_q  \cr
0_p  & 0  & 0_{p} & 0 \cr
0 & 0_q  & 0 & 0_{q} \cr\end{pmatrix}, \hspace{5mm}
\Psi_1(E_{22})
 =
\begin{pmatrix}
0_{p} & 0 & 0_{p} & 0\cr
0  & 0_{q} & 0 & 0_{q} \cr
0_{p} & 0_q & I_{p} & 0 \cr
0 & 0_{q} &0  & I_{q} \cr\end{pmatrix}, \\
\Psi_1(E_{12})
&=\begin{pmatrix} 0_{p} & 0 & I_{p} & 0\cr
0 & 0_{q} & 0 & 0_{q} \cr
0_{p} & 0 & 0_{p} & 0 \cr
0 & I_{q} & 0 & 0_{q} \cr\end{pmatrix}, \hspace{5mm}
\Psi_1(E_{21})
 =\begin{pmatrix} 0_{p} & 0 & 0_{p} & 0\cr
0 & 0_{q} & 0 & I_{q} \cr
I_{p} & 0 & 0_{p} & 0 \cr
0 & 0_{q} & 0 & 0_{q} \cr\end{pmatrix}.
\end{align*}
Consequently, $\Psi_1$ assumes the form
$$
\begin{pmatrix}
  a & b \\
  c & d
\end{pmatrix}
\mapsto
\begin{pmatrix}
aI_{p} & 0 & bI_{p} & 0\cr
0 & aI_{q} & 0 & cI_{q} \cr
cI_{p} & 0 & dI_{p} & 0 \cr
0 & bI_{q} & 0 & dI_{q} \cr\end{pmatrix}.
$$
After a permutation similarity,
$$
 V\Psi_1(A)V^{-1} = \begin{pmatrix}
 I_{p} \otimes A &0 \cr
 0 &  I_{q} \otimes A^\T \end{pmatrix}.
$$
Let $W=\big(I_k\oplus (B_{12} + C_{12})\big)^{-1}(U\oplus U)^{-1}V^{-1}\in \bM_{n(p+q)}(\IZ_2)$.  Then
 $\Phi_1$ assumes the form
$$
\Phi_1(A) = W\begin{pmatrix}
 I_{p} \otimes A &0 \cr
 0 &  I_{q} \otimes A^\T \end{pmatrix}W^{-1}.
 $$

Since $R\Phi_1(A)=\Phi_1(A)R$, we have
$$
(W^{-1}RW)\begin{pmatrix}
 I_{p} \otimes A &0 \cr
 0 &  I_{q} \otimes A^\T \end{pmatrix}
 =
\begin{pmatrix}
 I_{p} \otimes A &0 \cr
 0 &  I_{q} \otimes A^\T \end{pmatrix}(W^{-1}RW)\quad\text{for all $A\in \bM_2(\IZ_2)$.}
$$
Hence,
$$
W^{-1}RW = \begin{pmatrix}
 R_1 \otimes I_n &0 \cr
 0 &  R_2 \otimes I_{n} \end{pmatrix}
$$
for some invertible matrices $R_1\in \bM_p(\IZ_2)$ and $R_2\in \bM_q(\IZ_2)$.
Replacing $S$ by $S(W\oplus I_{r-n(p+q)})$, we see that $\Phi$ assumes the asserted form.

If $\bM_n(\IF)\neq \bM_2(\IZ_2)$ then,
by Theorem \ref{thm:J-homo}(a), we can assume that $\Phi_1$ also carries
the form  $\Phi_1(A)= W[(I_{p}\otimes A)\oplus (I_{q}\otimes A^\T)]W^{-1}$ for some invertible matrix $W\in \bM_{n(p+q)}(\IF)$.
 A similar argument derives that
 $W^{-1}RW=(R_1\otimes I_n)\oplus (R_2\otimes I_n)$,  and thus $\Phi$ assumes the stated form by replacing
 $S$ with $S(W\oplus I_{r-n(p+q)})$.

Finally, if $\Phi(I_n)=R$ is invertible, then $\Phi(A)=S[(R_1\otimes A)\oplus (R_2\otimes A^\T)]S^{-1}$ for all $A\in \bM_n(\IF)$,
as asserted.  If $\Phi$ sends rank one idempotents  to idempotents, then  $\Phi_0$ sends any rank one idempotent
to a nilpotent idempotent, and thus a zero matrix.  Since every matrix is a linear combination of rank one idempotent by Lemma \ref{P1P2},
 $\Phi_0$ is a zero map.
\end{proof}

\begin{remark}\label{rem:CxDZP}
In Theorem \ref{DZP-main},
more can be said about $\Phi_0$ if $\IF=\IR$ or $\IC$, namely, one can conclude that
$$
\Phi_0(X)^{\nu+1}=0\quad\text{for all $X\in \bM_n(\IF)$.}
$$
To see this, first note that we have $\Phi_0(A)^{\nu+1}=0$ for every diagonalizable $A$.
Now, for a given $A = (a_{ij}) \in \bM_n(\IF)$, we can write $A=B+C$ with $B = \diag(L, 2L, \dots, nL)$ for sufficiently large $L>0$
so that the $n$ Gershgorin disks of $C=A-B$
$$D_j = \{ z: |z-a_{jj}+ jL| \le \sum_{k\ne j} |a_{jk}|\}, \quad j = 1, \dots, n$$
are disjoint. For instance, this will hold if $L>2\sum_{k=1}^n |a_{jk}|$ for all $j = 1, \dots, n$.
In such a case the centers of any two disks $D_j$ and $D_\ell$ will satisfy 
$$
|(a_{jj}-jL)-(a_{\ell\ell}-\ell L)|\geq |(j-\ell)L|-|a_{jj}|-|a_{\ell\ell}|>  \sum_{k\ne j} |a_{jk}|+\sum_{k\ne \ell} |a_{\ell k}| 
$$ 
which are the sum of the radii of the two disks. So the two disks are disjoint.
Therefore, there is one eigenvalue in every disk $D_j$ for $j = 1, \dots, n$.
If $\IF = \IR$, then the complex eigenvalues of $C$ will occur in conjugate
pairs. Since the disjoint circular disks $D_1, \dots, D_n$ have centers
$a_{11} - L, a_{22} - 2L, \dots, a_{nn}-nL$ on the real line,
we see that $C$ has $n$ distinct real eigenvalues.
Thus, in both the real and complex cases, $B, C$ are diagonalizable.
Consequently,
$\Phi(B)^{\nu+1} = \Phi(C)^{\nu+1} = 0$.
Moreover, for $t = 2, \dots, \nu+1$, one can use the above
argument to show that $tB + C$ has $n$ disjoint Gershgorin disks,
and hence has $n$ distinct eigenvalues in $\IF$.
Thus,
$$ 0 = \Phi(tB+C)^{\nu+1} = (t\Phi(B) + \Phi(C))^{\nu+1}
= \sum_{k=0}^{\nu+1} t^k G_k(\Phi(B), \Phi(C)),$$
where $G_k(\Phi(B),\Phi(C))$ is the sum of ${\nu+1\choose k}$
matrices, and  each summand is a product of
$\nu+1$ matrices with $k$ of them equal to $\Phi(B)$ and the rest
equal to $\Phi(C)$.
In particular, $G_{\nu+1}(\Phi(B),\Phi(C)) = \Phi(B)^{\nu+1} = 0$ and
$G_0(\Phi(B), \Phi(C)) = \Phi(C)^{\nu +1}=0$.
So,
$$0 = \sum_{k=1}^{\nu} t^k G_k(\Phi(B),\Phi(C)), \quad t = 0, 2, \dots, \nu+1.$$
Now, the matrix-coefficient  polynomial in $t$ has degree at most $\nu$,
and has $\nu+1$ zeros. So,
$G_k(\Phi(B),\Phi(C)) = 0$ for all $k = 1, \dots, \nu$, and thus
$\Phi(A)^{\nu+1} = \Phi(B+C)^{\nu+1} = 0$.

\medskip
By the above discussion, we see that in Theorem \ref{DZP-main},
even if no extra information on $\IF$ is given,
one can conclude that
$\Phi_0(A)^{\nu+1} = 0$ whenever $A = B+C$ such that
$B, C$, $t_1 B + C, \dots, t_{\nu} B+C$ are diagonalizable for
$\nu$ distinct nonzero elements $t_1,\dots, t_{\nu} \in \IF$.
\end{remark}

In  the proof of Theorem \ref{DZP-main},
one can deduce from
\eqref{eq:PI_n} that the
 unital double zero product linear preserver $\Phi_1$
sends  idempotents to  idempotents.
When the underlying field
$\IF$ does \emph{not} have characteristic $2$,
a routine argument will  show that $\Phi_1$ is a Jordan homomorphism,
as well as a direct sum of a homomorphism and
an anti-homomorphism, and Theorem \ref{thm:alg-homo} applies.
However, the following example  tells us
that a  linear Jordan homomorphism between matrices over a field of characteristic $2$
does not necessarily  assume the form \eqref{oldform-2}.
Therefore, Theorem \ref{thm:J-homo}, together with a separate treatment to the case when $\bM_n(\IF)=\bM_2(\IZ_2)$,
is necessary in proving Theorem \ref{DZP-main}.

\begin{example}\label{eg:ch2-id-not}
Let $\IF$ be a field of characteristic $2$ and $\Phi: \bM_n(\IF)\to \bM_r(\IF)$ defined by
$$
A \mapsto \operatorname{trace}(A)I_r.
$$
For any $A,B\in \bM_n(\IF)$, observe that both
\begin{gather*}
\operatorname{trace}(AB + BA) = 2\operatorname{trace}(AB)=0, \quad\text{and} \\
\operatorname{trace}(A)\operatorname{trace}(B) + \operatorname{trace}(B)\operatorname{trace}(A)
= 2\operatorname{trace}(A)\operatorname{trace}(B)=0.
\end{gather*}
Hence, $\Phi$ is a  linear
Jordan homomorphism.
Since the trace of an idempotent in $\bM_n(\IF)$ is either zero or one (modulo $2$),
 $\Phi$ also sends idempotents to idempotents.
  But $\Phi$ does not assume the form \eqref{oldform-2}.

Note that
$\Phi$  does not send disjoint idempotents to disjoint idempotents,
and thus Theorem \ref{thm:J-homo} does not apply.
\end{example}

\section{Zero product preserving  maps into nilpotents}\label{s:3.3}

By Theorem \ref{main}, every zero product preserving additive map $\Phi: \bM_n(\mathbb{F}) \rightarrow \bM_r(\mathbb{F})$
assumes the form
\begin{align}\label{eq:stand-form-for-zp}
A \mapsto S(R\Phi_1(A) \oplus \Phi_0(A))S^{-1}= S(\Phi_1(A)R \oplus \Phi_0(A))S^{-1},
\end{align}
where $R, S$ are invertible matrices,   $\Phi_1: \bM_n(\mathbb{F})\to \bM_{nk}(\mathbb{F})$ is a unital ring homomorphism and
$\Phi_0 : \bM_n(\mathbb{F})\to \bM_{r-nk}(\mathbb{F})$ is a zero product preserving additive map sending matrices to nilpotent matrices.
With the discussion in Section \ref{s:0p}, we have a good understanding of
$\Phi_1$. In this section,
we focus on $\Phi_0$.
By Theorem \ref{main}, if $\Phi_0(I_n) = 0$,  then
$\Phi_0(\bM_n(\mathbb{F}))$ has trivial multiplication.

In \cite[Theorem 5.2]{BS-zp}, it is shown that every zero product preserving additive map
$\Phi: \bM_n(\mathbf{D})\to \bM_n(\mathbf{D})$ of matrices over a division ring
$\mathbf{D}$  either has a range with trivial multiplication, or
$\Phi(\cdot)=C\Psi(\cdot)=\Psi(\cdot)C$ for a ring endomorphism $\Psi$ and a  matrix $C$.
However, such conclusion may not hold for maps between matrices of different sizes.
For instance,  we can have the following example based on
{\cite[p.~310]{ORS01}} and {\cite[Example 2.5]{CKLW03}}.

\begin{example}\label{eg:kk2}
 Consider $\Phi:\bM_n(\mathbb{F})\to \bM_{r}(\mathbb{F})$
with $r\geq n+2$ and $n\neq 1$ defined by
$$
\begin{pmatrix}
a_{ij}
\end{pmatrix}
\mapsto
\begin{pmatrix}
    0      &  a_{11}  & \cdots &  a_{1n} & 0      & 0  & \cdots & 0\\
    0      &   0      & \cdots &   0     & a_{1n} & 0 & \cdots & 0\\
    \vdots & \vdots   & \ddots & \vdots  & \vdots & \vdots & \ddots &0  \cr
    0      &   0      & \cdots &    0    &  a_{nn}& 0 & \cdots & 0\\
    0      &   0      & \cdots &    0    &  0& 0 & \cdots & 0
    \end{pmatrix}.
$$
The linear map $\Phi$  preserves zero products.
Note that $\Phi(I_n)^{2}=0$, and thus any product of three elements in $\Phi(\bM_n(\IF))$ is zero.  Since
$\Phi(E)^2\neq 0$ with $E=E_{11}+ E_{1n}$,
the image of $\Phi$ has a
nontrivial multiplication.

We claim that $\Phi$ cannot be written as the form $C\Psi$
for any   $C$ in $\bM_{r}({\mathbb F})$ and any
homomorphism $\Psi:\bM_n({\mathbb F})\to \bM_{r}({\mathbb F})$.
Assume on the contrary that $\Phi=C\Psi$.  Then we get a contradiction since
\begin{align*}
\Phi(E)^2&=\Phi(E)C\Psi(E) =\Phi(E)C\Psi(E_{11}E)
    =\Phi(E)C\Psi(E_{11})\Psi(E)\\
    &    =\Phi(E)\Phi(E_{11})\Psi(E)
    =0\Psi(E)=0.
\end{align*}
\end{example}

We have the following.

\begin{proposition}\label{prop:goodhome-n=r}
Suppose that  $r \leq n+1$ and $n\neq 1$.
Let  $\Phi: \bM_n(\mathbb{F})\to \bM_r(\mathbb{F})$ be an additive zero product preserver
such that $\Phi(I_n)$ is a nilpotent matrix.
Then $\Phi(X)\Phi(Y) = 0$ for any $X, Y \in \bM_n(\IF)$.
\end{proposition}

By the above proposition and Theorem \ref{main}, we have the following
counterpart of \cite[Theorem 5.2]{BS-zp}.

\begin{corollary}
Suppose that  $r \leq n+1$ and $n\neq 1$.
Let  $\Phi: \bM_n(\mathbb{F})\to \bM_r(\mathbb{F})$ be an additive zero product preserver.
  \begin{enumerate}[\rm (a)]
    \item If $\Phi(I_n)$ is not a nilpotent, then $r=n$ or $r=n+1$, and
    $\Phi$  assumes the form
     \begin{align}\label{eq:goodform-r<n+2}
  A\mapsto \alpha S(A_\tau \oplus 0_{r-n}) S^{-1}
  \end{align}
    for some nonzero scalar $\alpha$, an
    invertible matrix $S$ in $\bM_r(\mathbb{F})$, and a unital ring endomorphism $\tau$ of $\IF$.
    Here, $A_\tau = (\tau(a_{ij}))$ if $A=(a_{ij})$.
    \item If $\Phi(I_n)$ is a nilpotent,  then the range of $\Phi$ always has trivial multiplication.
    In the case when the underlying field $\mathbb{F}$ is an infinite
field of characteristic $2$, we assume in addition that $\Phi$ is $\mathbb{F}$-linear.
  \end{enumerate}
\end{corollary}

To prove Proposition \ref{prop:goodhome-n=r}, we need the following lemma.
It provides us a sufficient and necessary condition for
$\Phi_0(\bM_n(\mathbb{F}))$ having trivial multiplication.

\begin{lemma}
\label{lem:idempotents2square0}
Let  $\Phi: \bM_n(\mathbb{F})\to \bM_r(\mathbb{F})$ be an additive zero product preserver.  When $\mathbb{F}$ is an infinite
field of characteristic $2$, we assume in addition that $\Phi$ is $\mathbb{F}$-linear.
 The range of $\Phi$ has trivial multiplication exactly when
  $\Phi$ sends every  scalar multiple of a rank one idempotent to a square zero element.
\end{lemma}
\begin{proof}
We verify the sufficiency only.
By Lemma \ref{P1P2}, for every $X,Y$ in $\bM_n(\IF)$ we can write
their product as a linear combination of idempotents,
$XY=\sum_j \beta_j E_j$, say.
In the case when $2$ is invertible in $\mathbb{F}$, we see that each scalar
$\beta= \left(\dfrac{\beta +1}{2}\right)^2 - \left(\dfrac{\beta -1}{2}\right)^2$.
In the case when $\mathbb{F}$ is a finite field of characteristic $2$,
the map $\beta\mapsto \beta^2$ is  injective, and thus
bijective, from $\mathbb{F}$ onto $\mathbb{F}$.  Thus in both cases we can
assume that $\beta_k=\alpha_k^2-\gamma_k^2$ for some $\alpha_k, \gamma_k$ in $\mathbb{F}$ for all $k$.

If $\Phi$ is assumed additive and $\mathbb{F}$ is not an infinite field of
characteristic $2$, then with \eqref{iab=ab} we have
\begin{align*}
\Phi(X)\Phi(Y) &=\Phi(I_n)\Phi(XY) = \sum_j \Phi(I_n)\Phi((\alpha_j^2 -\gamma_j^2) E_j)\\
&= \sum_j \Phi(I_n)\Phi((\alpha_j E_j)^2)-\sum_j  \Phi(I_n)\Phi((\gamma_j E_j)^2)\\
&= \sum_j \Phi(\alpha_j E_j)^2 -\sum_j \Phi(\gamma_j E_j)^2=0.
\end{align*}
For the exceptional case that $\IF$ is an infinite field of characteristic $2$, with the linearity of $\Phi$
it follows from \eqref{iab=ab} that
\begin{align*}
\Phi(X)\Phi(Y)&=\Phi(I_n)\Phi(XY)=\Phi(I_n)\Phi(\sum_j \beta_j E_j)\\
&= \sum_j \beta_j\Phi(I_n)\Phi(E_j)= \sum_j \beta_j\Phi(E_j)^2=0.
\end{align*}
\vskip -.4in \end{proof}

\medskip
\begin{proof}[Proof of Proposition \ref{prop:goodhome-n=r}]
Let $\Phi(I_n)$ be a nilpotent.  Suppose on contrary that $\Phi(\bM_n(\IF))$ does not have trivial multiplication.
By Lemma \ref{lem:idempotents2square0},
$\Phi(\alpha E)^2\neq 0$ for a rank one idempotent $E$ in $\bM_n(\IF)$ and $\alpha\neq 0$ in $\mathbb{F}$.
Let $\{e_1,\ldots, e_n\}$ be a  basis of $\mathbb{F}^n$ consisting of
eigenvectors of $E$ such that $Ee_1=e_1$ and $Ee_j=0$ for $j=2,\ldots, n$.
In this setting, we can assume  $E=  E_{11}$.

Because $\Phi$ preserves zero products,
\begin{align}\label{eq:ijkl0}
\Phi(\alpha E_{ij})\Phi(\alpha E_{kl})=0, \quad \text{whenever $j\neq k$, and}\ i,j, k,l = 1, \ldots, n.
\end{align}
Observe also that
$$
(\alpha E_{11} + \alpha E_{1j})(\alpha E_{11} - \alpha E_{j1}) =0
$$
implies
\begin{align}\label{eq:1jj1}
\Phi(\alpha E_{1j})\Phi(\alpha E_{j1}) = \Phi(\alpha E_{11})^2 \neq0, \qquad j=1,\ldots, n.
\end{align}

Since $\Phi(I_n)$ is a nilpotent, $\Phi(\alpha E_{11})$ is a nilpotent as well by Lemma \ref{lem:zp}(b).
  After a similarity transformation, we can
assume that $\Phi(\alpha E_{11})=J_1\oplus\cdots \oplus J_m$ is a direct sum of its  Jordan blocks, all of
which have zero diagonals.  Since $\Phi(\alpha E_{11})^2\neq 0$, we can
further assume that $J_1$ is of size at least $3$; namely,
$$
J_1 = \begin{pmatrix}
     0  &  1       & \cdots &  0\\
          &        & \ddots &    \\
     0  &  0        & \cdots &  1\\
      0  &  0         & \cdots &  0\\
    \end{pmatrix}.
$$
Since $E_{1j}E_{11}=E_{11}E_{j1}=0$, we see that the first and the second  columns  of  $\Phi(\alpha E_{1j})$ are zero columns, and
the second and the third rows of $\Phi(\alpha E_{j1})$ are zero rows for $j=2,\ldots, n$.

Denote by $R_j$ the first row of  $\Phi(\alpha E_{1j})$, and by $C_j$
the third column of $\Phi(\alpha E_{j1})$ for $j=2,3, \ldots, n$.
Let
$$
R= \begin{pmatrix}
  R_2\\
  R_3 \\
  \vdots \\
  R_n
\end{pmatrix}_{(n-1)\times r}
\quad\text{and}\quad
C= \begin{pmatrix}
  C_2 &   C_3 &  \cdots &
  C_n
\end{pmatrix}_{r\times (n-1)}.
$$
The conditions \eqref{eq:ijkl0} and \eqref{eq:1jj1} tell us that $R_iC_j =1$ whenever $i=j$, and $0$ whenever $i\neq j$.
In other words,  $RC=I_{n-1}$.
Note that the first and second columns of $R$ are both the
zero columns.  On the other hand, since the third row of $C$ is the zero row, we can replace the third column of $R$ by the zero column to get
a new $(n-1)\times r$ matrix $R'$  such that $R'C=RC=I_{n-1}$.
Therefore, $R'$ has rank at least $n-1$. Since the first three columns of $R'$ are zero, we have
$r-3\geq n-1$.
This contradiction establishes the assertion.
\end{proof}

The unital ring homomorphism from $\IF$ into $\bM_2(\IF)$ given by \eqref{eq:derivation} in Example \ref{eg:ringhom}   shows that Proposition \ref{prop:goodhome-n=r} does not hold  when $r=2n =2$.

The following theorem demonstrates the structure of the range space
$\bV=\Phi(\bM_n(\IF))$ when $\Phi: \bM_n(\IF) \rightarrow \bM_{r}(\IF)$ is a linear map such that
$\Phi(A)\Phi(B)=0$ for all $A,B\in \bM_n(\IF)$.
In particular,  $\Phi(\bM_n(\IF))$ has
dimension at most $r^2/4$ in this case.

\begin{theorem} \label{main3}
Let $\bV$ be a vector subspace of $\bM_r(\IF)$.
\begin{itemize}
\item[{\rm (a)}]
Suppose that there is an invertible matrix $S$ in $\bM_r(\IF)$ such that
$S^{-1}\bV S$ consists of matrices of the form
\begin{equation}\label{Zform}
\begin{pmatrix}
0_p & Z_{12} & XB_1 \cr
0_p & 0_p & 0\cr
0 & B_2 Y & 0_q\cr
\end{pmatrix} \quad \text{with $Z_{12}\in \bM_{p}(\IF)$, $X \in \bM_{p,k_1}(\IF)$
 and $Y \in \bM_{k_2,p}(\IF)$,}
\end{equation}
for some nonnegative integers  $p, k_1, k_2$ with $k_1 + k_2 \le q=r-2p$, and fixed matrices
$B_1 \in \bM_{k_1,q}(\IF)$ and $B_2 \in \bM_{q,k_2}(\IF)$ 
with $B_1B_2 = 0_{k_1,k_2}$.
Then $\bV$ has trivial multiplication.
Moreover, the dimension of $\bV$ is at most $p(r-p)\leq r^2/4$.

\item[{\rm (b)}] If $\IF$ has more than $r/2$
elements and $\bV$ has trivial multiplication, then we can represent $\bV$ as in \eqref{Zform} such that
$p$ is the maximal rank of matrices in $\bV$, and $B_1, B_2$ are full rank matrices.
\end{itemize}
\end{theorem}

\noindent
\it Proof. \rm (a) Since $B_1 B_2=0$, the product $ZZ'=0$ for all $Z,Z'\in \bV$.
Moreover, since a matrix in $\bV$ is determined by
$Z_{12} \in\bM_p(\IF)$, $X \in \bM_{p,k_1}(\IF)$ and $Y\in \bM_{k_2,p}(\IF)$,
 the dimension of $\bV$ is bounded by
$$p^2 + p(k_1+k_2) \le p^2 + pq = p(r-p) \le r^2/4.$$

\rm (b)
Let  $Y \in \bV$ have the maximal rank $p$
among the matrices in $\bV$. Because $Y^2 = 0$,
we may apply a similarity transform and assume that
$$Y = \begin{pmatrix} 0_p & I_p & 0 \cr
0_p & 0_p & 0 \cr
0 & 0 & 0_q\cr\end{pmatrix}$$
with $2p+q = r$.

Let $Z\in \bV$.
Since $YZ = ZY = 0$,
it follows that
\begin{equation}\label{Zform2}
Z = \begin{pmatrix}  0_p & Z_{12} & Z_{13} \cr
0_p & 0_p & 0 \cr
0 & Z_{32} & Z_{33} \cr\end{pmatrix}.
\end{equation}
We claim that $Z_{33}=0$ when $q\geq 1$.
Note that $tY + Z$ has rank at most $p$ for all $t \in \IF$.
Let $G(t,j,k)$ be a submatrix of
$\begin{pmatrix} tI_p + Z_{12} & Z_{13} \cr Z_{32} & Z_{33}\cr\end{pmatrix}$
of size $p+1$ including the leading principal  submatrix $tI_p + Z_{12}$ and the
$(j,k)$ entry, $\gamma_{jk}$ say, of $Z_{3,3}$  as the $(p+1,p+1)$ entry.
Then, 
$\det(G(t,j,k))$
is a polynomial  with leading term $\gamma_{jk} t^p$,
and $\det(G(t,j,k)) = 0$
 for all $t \in \IF$.
 Since $\IF$ has more than   $r/2 \ge p+1/2$ elements, namely, $\IF$ has at least $p+1$ elements,
 $\det(G(t,j,k))$ is the zero polynomial and $\gamma_{jk} = 0$.
Because this is true for any entry $\gamma_{jk}$ of $Z_{33}$, we see that
 $Z_{33} = 0_q$.
Consequently, every $Z$ in $\bV$ has the form
(\ref{Zform2}) with $Z_{33} = 0_q$ and
$Z_{13}Z_{32}=0_p$.

Let $\{r_1,\ldots r_{k_1}\}$ be a basis for the vector space spanned by the row vectors
of matrices $Z_{13}$ with all $Z \in \bV$, and $\{c_1,\ldots, c_{k_2}\}$ be
a basis for the vector space spanned by the column vectors of matrices in  $Z_{32}$
with all $Z \in \bV$.
Let  $B_1 \in \bM_{k_1,q}(\IF)$ have row
vectors $r_1, \ldots, r_{k_1}$, and  $B_2\in \bM_{q, k_2}(\IF)$ have column vectors $c_1,\ldots, c_{k_2}$.
We can think of all $r_i^\T$ and all $c_j$ as vectors in $\IF^q$ with $r_i c_j=0$.
In other words, $B_1B_2=0$.  It follows that the nullity of the rank $k_1$ matrix $B_1$ 
is at least $k_2$, which is the rank of $B_2$.
The rank–nullity theorem implies that $q-k_1 \geq k_2$.

For any $Z\in \bV$, we have $Z_{13} = XB_1$ for some $X \in \bM_{p,k_1}(\IF)$
and $Z_{32}   =B_2 Y$ for some $Y \in \bM_{k_2,q}(\IF)$.
Hence, $Z \in \bV$ has the asserted form.
\qed

Our proof used ideas in \cite{SM09}, in which
the structure of a vector subspace $\bW$ of $\bM_r(\IF)$ carrying trivial Jordan product is given. 
In particular, the following was obtained.

\begin{proposition}[{\cite[Theorem 4]{SM09}}]\label{prop:ZJP}
  Let $\bW$ be a vector subspace of $\bM_r(\IF)$ consisting of square zero matrices.
   Suppose that $\IF$ has more than $r/2$ elements,
   and all matrices in $\bW$ have rank at most $p$.
   Then $\bW$ has dimension at most $p(r-p)\leq r^2/4$.
\end{proposition}

Let $\IF=\IR$ or $\IC$, and
let $\Phi: \bM_n(\IF)\to \bM_r(\IF)$ be a real or complex linear map preserving double zero products.
  If $\Phi(I_n)=0$, by Remark \ref{rem:CxDZP} we see that
$\Phi(A)^2=0$ for all $A\in \bM_n(\IF)$.  It then follows from Proposition \ref{prop:ZJP} that
the range space $\bW=\Phi(\bM_n(\IF))$ has dimension at most $r^2/4$.


\section*{Acknowledgment}

C.-K. Li is an affiliate member of the Institute for Quantum Computing,
University of Waterloo. His research is supported by
Simons Foundation Grant 851334. This
research
started during his academic visit to Taiwan in 2018, which was supported by grants from Taiwan MOST.
He would like to express
his gratitude to the hospitality of several institutions there, including the
Academia Sinica,
National Chung Hsing University,  National Sun
Yat-sen University, and National Taipei University of Technology.

M.-C. Tsai, Y.-S. Wang and N.-C. Wong are supported by Taiwan MOST grants 110-2115-M-027-002-MY2,
111-2115-M-005-001-MY2 and  110-2115-M-110-002-MY2,
respectively.


\begin{thebibliography}{WWW}

\bibitem{ABEV09}
 J. Alaminos, M. Bre\v{s}ar, J. Extremera and A. R. Villena,
Maps preserving zero products, \emph{Studia Math.}, \textbf{193} (2009), 131--159.



\bibitem{Bresar12}
M. Bre\v{s}ar, Multiplication algebra and maps determined by zero products,
\emph{Linear Algebra Appl.}, \textbf{60}:7 (2012),  763--768.


\bibitem{BS-zp}
M. Bre\v{s}ar and P. \v{S}emrl,
Zero product preserving maps on matrix rings over division rings, Contemp. Math.,
in ``Linear and multilinear algebra and function spaces'', 195--213,
\emph{Contemp.\ Math.}, \textbf{750}, Centre Rech.\ Math.\ Proc., Amer.\ Math.\ Soc.,  Providence, RI, 2020.

\bibitem{BFGMP08}
M. Burgos, F. J. Fern\'{a}ndez-Polo, J. J. Garc{\'e}s, J. Mart\'{\i}nez Moreno and A. M. Peralta,
Orthogonality preservers in $C^*$-algebras, JB$^*$-algebras and JB$^*$-triples, \emph{J. Math. Anal. Appl.},
\textbf{348} (2008), 220--233.

\bibitem{CLT87}
G.-H. Chan,
M.-H. Lim and K.-K. Tan,
Linear preservers on matrices, \emph{Linear Algebra Appl.}, \textbf{93} (1987), 67--80.


\bibitem{CKLW03} M. A. Chebotar, W.-F. Ke, P.-H. Lee and N.-C. Wong, Mappings
        preserving zero products, \emph{Studia Math.}, \textbf{155}:1 (2003), 77--94.

\bibitem{CLP} W.-S. Cheung, C.-K. Li, and Y.-T. Poon,
Isometries between matrix algebras, \emph{J. Aust. Math. Soc.}, \textbf{77} (2004), 1--16.



\bibitem{Erdos67}
J. A. Erdos, On products of idempotent matrices,
\emph{Glasg. Math. J.}, \textbf{8} (1967) 118--122.




\bibitem{GLS} A. Guterman, C.-K. Li, and P. \v{S}emrl,
Some general techniques on linear preserver problems,
\emph{Linear Algebra Appl.}, \textbf{315} (2000), 61--81.



\bibitem{H56}
I. N. Herstein,
Jordan homomorphisms,  \emph{Trans.\ Amer.\ Math.\ Soc.},
\textbf{81}:2 (1956), 331--341.






\bibitem{JR50}
N. Jacobson and C. E. Rickart, Jordan homomorphisms of rings,  \emph{Trans.\ Amer.\ Math.\ Soc.},
\textbf{69}:3 (1950), 479--502.

\bibitem{LW13} A. T.-M. Lau and N.-C. Wong,
Orthogonality and disjointness preserving linear maps between Fourier and Fourier-Stieltjes algebras of locally compact groups,
\emph{J. Funct.\ Anal.}, \textbf{265}:4 (2013),  562--593.



\bibitem{LP} C.-K. Li and S. Pierce,
Linear preserver problems, \emph{Amer.\ Math.\ Monthly}, \textbf{108} (2001), 591--605.


\bibitem{Li02} C.-K. Li, L. Rodman, and P. \v{S}emrl,
Linear transformations between matrix spaces that map one rank specific set into another,
\emph{Linear Algebra Appl.}, \textbf{357} (2002), 197--208.

\bibitem{LTWW20}
C.-K. Li, M.-C. Tsai, Y.-S. Wang, and N.-C. Wong,
Nonsurjective maps between rectangular matrix spaces
preserving disjointness, triple products, or norms,
\emph{J. Operator Theory}, \textbf{83} (2020),   27--53.

\bibitem{LWWT-Idem} C.-K. Li, M.-C. Tsai, Y.-S. Wang and N.-C. Wong,
Linear maps
preserving  matrices annihilated by a fixed  polynomial,
http://arxiv.org/pdf/2302.11170v1.pdf

\bibitem{LT92}
C.-K. Li and N.-K. Tsing,
Linear preserver problems: A brief introduction and some special techniques,
\emph{Linear Algebra Appl.}, \textbf{162-164} (1992),  217--235.

\bibitem{LZ} C.-K. Li and F. Zhang,
Eigenvalue continuity and Gershgorin Theorem,
Electronic J. Linear Algebra. \textbf{35} (2019), 619-625.


\bibitem{LCLW18}
J.-H. Liu, C.-Y. Chou, C.-J. Liao and N.-C. Wong,
Linear disjointness preservers of operator algebras and related structures, \emph{Acta Sci.\ Math.\ (Szeged)},
 \textbf{84} (2018),  277--307.


\bibitem{Monlar} L. Monlar, \emph{Selected
Preserver Problems on Algebraic Structures of Linear Operators and on Function Spaces},
Springer-Verlag, Berlin, 2007.



\bibitem{ORS01}
M. Omladi\v c, H. Radjavi, P. \v Semrl,
Preserving commutativity, \emph{J. Pure Appl.\ Algebra}, \textbf{156} (2001),
309--328.


\bibitem{Semrl06}
P. \v{S}emrl,
Maps on matrix spaces, \emph{Linear Algebra Appl.}, \textbf{413} (2006), 364--393.



\bibitem{SM09}
L. G. Sweet and J. A. MacDougal,
The maximum dimension of a subspace of nilpotent
matrices of index 2,
\emph{Linear Algebra Appl.},
\textbf{431} (2009), 1116-–1124.

\bibitem{V}
 R.S. Varga, \emph{Gershgorin and his circles},
 Springer Series in Computational Mathematics 36, Springer-Verlag, Berlin, 2004

\bibitem{Wong05}
 N.-C. Wong,
 Triple homomorphisms of operators algebras,
 \emph{Southeast Asian Bull.\ Math.}, \textbf{29} (2005), 401--407.

\bibitem{W-zpp}
  N.-C. Wong,
  Zero product preservers of $C^*$-algebras,
 \emph{Contemp.\ Math.}, \textbf{435} (2007), 377--380.


\bibitem{SZbook75}
O. Zariski and P. Samuel, \emph{Commutative algebra, I}, Graduate Texts in Mathematics  28, Springer-Verlag, Berlin, New York, 1975.

\bibitem{ZTC06}
X. Zhang, X.-M. Tang and C.-G. Cao, \emph{Preserver problems on spaces of matrices}, Science Press, Beijing, 2006.



\end{thebibliography}
\end{document}